\documentclass[11pt,a4paper, reqno]{article}
\usepackage[utf8]{inputenc}
\usepackage{amsfonts}
\usepackage{amssymb}
\usepackage{enumerate}
\usepackage{graphicx}
\usepackage{mathtools}
\usepackage{amsmath, amsthm}
\usepackage{tikz}
\usepackage{color}
\usepackage[affil-it]{authblk}
\usepackage[T1]{fontenc}
\usepackage[sort,numbers]{natbib}
\usepackage{bbm}
\usepackage[a4paper, top = 1in, bottom = 1in, left = 1in, right = 1in]{geometry}
\usepackage{amsmath}
\usepackage{amsfonts}
\usepackage{amssymb}
\usepackage{bbm}
\usepackage{mathrsfs}
\usepackage[utf8]{inputenc}
\usepackage[english]{babel}
\usepackage{hyperref}
\usepackage{relsize}
\usepackage{accents}
\usepackage[noabbrev]{cleveref}
\usepackage{appendix}
\usepackage{lipsum}
\usepackage[color = white]{todonotes}

\newcommand\blfootnote[1]{%
  \begingroup
  \renewcommand\thefootnote{}\footnote{#1}%
  \addtocounter{footnote}{-1}%
  \endgroup
}

\makeatletter
  \@addtoreset{chapter}{part}
  \@addtoreset{@ppsaveapp}{part}
\makeatother

\long\def\/*#1*/{}

\usepackage{palatino}

\makeatletter
\newsavebox\myboxA
\newsavebox\myboxB
\newlength\mylenA

\newcommand*\xoverline[2][0.75]{%
    \sbox{\myboxA}{$\m@th#2$}%
    \setbox\myboxB\null
    \ht\myboxB=\ht\myboxA%
    \dp\myboxB=\dp\myboxA%
    \wd\myboxB=#1\wd\myboxA
    \sbox\myboxB{$\m@th\overline{\copy\myboxB}$}
    \setlength\mylenA{\the\wd\myboxA}
    \addtolength\mylenA{-\the\wd\myboxB}%
    \ifdim\wd\myboxB<\wd\myboxA%
       \rlap{\hskip 0.5\mylenA\usebox\myboxB}{\usebox\myboxA}%
    \else
        \hskip -0.5\mylenA\rlap{\usebox\myboxA}{\hskip 0.5\mylenA\usebox\myboxB}%
    \fi}
\makeatother

\usepackage{environ}
\NewEnviron{eq}{%
\begin{equation}\begin{split}
  \BODY
\end{split}\end{equation}
}

\usepackage{hyperref}
\hypersetup{
  colorlinks,
  linkcolor=blue,
  citecolor=black,
  filecolor=black,
  urlcolor=black,
  pdftitle={},
  pdfauthor={S. Dhara, R. van der Hofstad, J. van Leeuwaarden, S. Sen},
  pdfcreator={S. Dhara},
  pdfsubject={},
  pdfkeywords={}
}
\numberwithin{equation}{section}

\def\sss{\scriptscriptstyle}

\newcommand{\floor}[1]{\ensuremath{\left\lfloor #1 \right\rfloor}}
\newcommand{\ind}[1]{\ensuremath{\mathbbm{1}_{\left\{#1\right\}}}}
\newcommand{\pto}{\ensuremath{\xrightarrow{\mathbbm{P}}}}

\newcommand{\PR}{\ensuremath{\mathbbm{P}}}
\newcommand{\E}{\ensuremath{\mathbbm{E}}}

\newcommand{\e}{\ensuremath{\mathrm{e}}}

\newcommand{\OP}{\ensuremath{O_{\sss\PR}}}
\newcommand{\oP}{\ensuremath{o_{\sss\PR}}}
\newcommand{\thetaP}{\ensuremath{\Theta_{\sss\PR}}}

\newcommand{\dif}{\mathrm{d}}
\newcommand{\bld}[1]{\boldsymbol{#1}}

\newcommand{\NR}{\mathrm{MNR}_n(\bld{w})}
\newcommand{\sNR}{\mathrm{SNR}_n(\bld{w})}
\newcommand{\mNR}{\mathrm{MNR}_n(\bld{w})}

\newcommand{\rNR}{\mathrm{MNR}_n}
\newcommand{\rsNR}{\mathrm{SNR}_n}

\newcommand{\bw}{\bld{w}}

\newcommand{\sC}{\mathscr{C}}

\newcommand{\sCi}{\mathscr{C}_{\sss (i)}}

\newcommand{\tE}{\tilde{\mathbbm{E}}}
\newcommand{\tPR}{\tilde{\mathbbm{P}}}

\newcommand{\sF}{\mathscr{F}}

\newcommand{\cf}{c_{\sss \mathrm{F}}}
\newcommand{\barcf}{\bar{c}_{\sss \mathrm{F}}}

\newcommand{\sB}{\mathscr{B}}
\newcommand{\cS}{\mathcal{S}}

\newtheorem{theorem}{Theorem}
\newtheorem{algo}[theorem]{Algorithm}
\newtheorem{lemma}[theorem]{Lemma}
\newtheorem{proposition}[theorem]{Proposition}

\newtheorem{assumption}[theorem]{Assumption}
\newtheorem{remark}[theorem]{Remark}
\newtheorem{fact}[theorem]{Fact}

\numberwithin{theorem}{section}

\def\qed{ \hfill $\blacksquare$}

\newcommand{\cF}{\mathscr{F}}
\newcommand{\cG}{\mathcal{G}}

\newcommand{\cN}{\mathcal{N}}

\newcommand{\cX}{\mathcal{X}}

\newcommand{\Var}{\mathrm{Var}}
\newcommand{\Cov}{\mathrm{Cov}}
\newcommand{\cT}{\mathcal{T}}
\newcommand{\sCa}{\mathscr{C}^a}

\newcommand{\eqn}[1]{\begin{equation} #1 \end{equation}}
\newcommand{\eqan}[1]{\begin{align} #1 \end{align}}

\newcommand{\vep}{\varepsilon}
\newcommand{\Prob}{{\mathbb P}}

\newcommand{\op}{o_{\sss \Prob}}

\newcommand{\nn}{\nonumber}

\newcommand{\Poi}{{\sf Poisson}}

\newcommand{\seta}{\eta_s}
\newcommand{\per}{\pi}

\newcommand{\percn}{\pi_n}

\newcommand{\ww}{\bar{w}}
\newcommand{\bell}{\bar{\ell}}

\begin{document}
\title{Barely supercritical percolation on \\
Poissonian scale-free networks}
\author{Souvik Dhara$^1$ and Remco van der Hofstad$^2$}

\maketitle
\begin{abstract}
We study the giant component problem slightly above the critical regime for percolation on Poissonian random graphs in the scale-free regime, where the vertex weights and degrees have a diverging second moment. 
Critical percolation on scale-free random graphs have  been observed to have incredibly subtle features that are markedly different compared to those in random graphs with converging second moment. In particular, the critical window for percolation depends sensitively on whether we consider {\em single-} or {\em multi-edge} versions of the Poissonian random graph.

In this paper, and together with our companion paper~\cite{BDH18}, we build a bridge between these two cases. Our results characterize the part of the barely supercritical regime where the size of the giant components are approximately same for the single- and multi-edge settings.  
The methods for establishing concentration of giant for the single- and multi-edge versions are quite different. While the analysis in the multi-edge case is based on scaling limits of exploration processes, the  single-edge setting requires identification of a core structure inside certain  high-degree vertices that forms the giant component.
\end{abstract}
\blfootnote{\emph{Emails:} 
 \href{mailto:sdhara@mit.edu}{sdhara@mit.edu},
 \href{mailto:r.w.v.d.hofstad@tue.nl}{r.w.v.d.hofstad@tue.nl}} 
\blfootnote{$^1$Department of Mathematics, Massachusetts Institute of Technology}
\blfootnote{$^2$Department of Mathematics and Computer Science, Eindhoven University of Technology}

\noindent 
\section{Introduction}

\subsection{Background and motivation}
Given a (random) graph, {\em percolation} refers to
deleting each edge with some probability $1-\pi$,  independently. 
Percolation gives an elementary model for the structural transition in the connectivity of  large-network architectures as the value of $\pi$ increases. 
This structural connectivity phase transition is widely recognized as a central problem, arising e.g.\ in the study of epidemics on real-world networks,  or the robustness of networks such as the Internet when the edges of the underlying network experience random breakdown~\cite{Bar16,Newman-book}. 
Empirical studies on a wide variety of real-world systems such as the World-Wide Web, social networks suggest that scale-free networks, with degree distribution having diverging second-moment, are ubiquitous (see \cite{Bar16} and the references therein). 
As a result, percolation on scale-free  networks has attracted enormous attention, both in the rigorous as well as in the applied literature. 
Mathematically, percolation processes on these random graphs display quite different behavior, since scale-free networks have a highly {\em inhomogeneous} degree distribution, and the vertices with extremal degrees play a special role in the formation of large components.

Given a graph sequence $(G_n)_{n\geq 1}$ with $G_n$ having $n$ vertices consider performing percolation with probability $\pi_n$, where $\pi_n$ may depend on the graph size. 
We say that percolation is \emph{supercritical} if there is a \emph{giant} component in the sense that, with high probability, there is a component whose size is much larger than that of all other components.  

One of the well-known features of scale-free networks is that these networks are \emph{robust} under random edge-deletion, i.e., for any sequence $(\percn)_{n\geq 1}$ with $\liminf_{n\to\infty} \percn > 0$, the graph obtained by applying percolation with probability $\percn$ remains supercritical, and the giant component even contains a positive proportion of the vertices of the graph. 
This feature has been studied experimentally in~\cite{AJB00}, using heuristic arguments in~\cite{CEbAH00,CNSW00,DGM08,CbAH02}, and mathematically in~\cite{BR03}. 
Thus, in order to pinpoint for what values of $\pi_n$ the giant component appears, one needs to take $\pi_n \to 0$ the network size $n\to\infty$. This identifies the central question in this field:
\begin{quote}
    What is the right scaling of $\pi_n$ with $n$ in order to see a giant component?
\end{quote}
For scale-free networks, this question has recently been studied for the multi-edge version of the configuration model \cite{DHL19} and inhomogeneous random graphs where multi-edges are not allowed \cite{BDH18}. 

The main motivation of this paper stems from  observations in \cite{BDH18,DHL19}. 
Suppose that the weight distribution for generating the  random graph satisfies a power-law distribution with exponent $\tau \in (2,3)$ (see Assumption~\ref{assumption-NR} below for the precise conditions). 
It was shown in \cite{BDH18} that there exists an explicit constant $\lambda_c\in (0,\infty)$ satisfying the following: 
\begin{enumerate}[(1)]
    \item When the percolation parameter satisfies $\percn = \lambda n^{-(3-\tau)/2}$ and $\lambda \in (0, \lambda_c)$, the critical behavior is observed with largest components having size of order $n^{\beta}$ with $\beta=(\tau^2-4\tau+5)/[2(\tau-1)]\in[\sqrt{2}-1, \tfrac{1}{2})$, and the re-scaled vector of ordered component sizes have a non-degenerate scaling limit.
    \item When $\percn = \lambda n^{-(3-\tau)/2}$ and $\lambda \in (\lambda_c,\infty)$ is fixed, then a unique tiny giant component $\sC_{\sss (1)}^*$ emerges, satisfying $n^{-1/2} |\sC_{\sss (1)}^*| \xrightarrow{\sss \PR}   \zeta_1^{\lambda}$ for some $\zeta_1^{\lambda}>0$.
\end{enumerate}
On the other hand, for the configuration model, it was shown in \cite{DHL19} that a giant component emerges for much smaller $\pi_n$ values, more precisely, for $\pi_n \gg n^{-(3-\tau)/(\tau-1)}$. In particular, if $\sC_{\sss (1)}$ denotes the giant component for the configuration model, then for $\percn = \lambda n^{-(3-\tau)/2}$ and $\lambda \in (0,\infty)$, we have $n^{-1/2} |\sC_{\sss (1)}| \xrightarrow{\sss \PR}  \zeta_2^{\lambda}$. 
The description of $\zeta_1^{\lambda}$ is markedly different from $\zeta_2^{\lambda}$. While $\zeta_2^{\lambda}$ can be calculated explicitly in terms of the degree distribution, $\zeta_1^{\lambda}$ is characterized in terms of an implicit limit of  certain integrals of  survival probability functions of multi-type branching processes.
Since the later branching process has a non-rank-one kernel, one  
expects the implicit constant $\zeta_1^{\lambda}$ to be different from  $\zeta_2^{\lambda}$ even when we set the degrees and the weights to be equal. 
However, when $\liminf_{n\to\infty} \percn >0$, one would expect that that the asymptotic sizes of the giant components should not be dictated by such single-multi-edge considerations, since the local-weak limits are the same (in fact one can calculate the size of the giant component using the framework of general inhomogeneous random graphs of Bollob{\'a}s, Janson, and Riordan~\cite{BJR07}).
This raises the question: 
\begin{quote}
    When does the size of the giant become  identical for the models with single and multi-edge settings? 
\end{quote}
In this paper, we investigate this question further. To have a common model for the single- and multi-edge settings, we consider the Poissonian random graph or Norros-Reittu model~\cite{NR06} (see Section~\ref{sec-RG-mod}). 
Our main results analyze the size of the giant component for the barely supercritical phase for percolation on both the single and multi-edge settings. The analysis shows that, if $\percn = \lambda_n n^{-(3-\tau)/2}$ and $\lambda_n\to\infty$ (and also $\percn \to 0$), then there is no asymptotic difference between the sizes of the giant components in the single- and multi-edge settings. 

\subsection{Notation} 
\label{sec:notation}
To describe the main results of this paper, we need some definitions and notations. 
We use the standard notation of $\xrightarrow{\sss\PR}$ and $\xrightarrow{\sss d}$ to denote convergence in probability and in distribution, respectively.
We often use the Bachmann-Landau notation $O(\cdot)$, $o(\cdot)$, $\Theta(\cdot)$ for large $n$ asymptotics of real numbers.
We write $a_n \asymp b_n$ as a shorthand notation for $a_n/b_n = (1+o(1))$.
 A sequence of events $(\mathcal{E}_n)_{n\geq 1}$ is said to occur with high probability~(whp) with respect to the probability measures $(\mathbbm{P}_n)_{n\geq 1}$  when $\mathbbm{P}_n\big( \mathcal{E}_n \big) \to 1$. 
For (random) variables $X_n$ and $Y_n$, define $X_n = O_{\sss\mathbbm{P}}(Y_n)$ when  $ ( |X_n|/|Y_n| )_{n \geq 1} $ is a tight sequence; $X_n =o_{\sss\mathbbm{P}}(Y_n)$ when $X_n/Y_n  \xrightarrow{\sss\PR} 0 $; $X_n =\thetaP(Y_n)$ if both $X_n=\OP(Y_n) $ and $Y_n=\OP(X_n)$. 
We also use $C,C',C_1,C_2$ etc.\ as generic notations for positive constants whose values can be different in different equations.

Fix $\tau \in (2,3)$.  Throughout this paper, we write  
	\begin{equation}\label{eqn:notation-const}
 	\alpha= 1/(\tau-1),\qquad  \eta=(3-\tau)/(\tau-1), \qquad \seta = (3-\tau)/2.
	\end{equation}

\subsection{The Poissonian random graph}
\label{sec-RG-mod}
Let $\bw=(w_i)_{i\in [n]}$ be a given set of weights on the vertex set $[n]:= \{1,\dots, n\}$, and let $\ell_n = \sum_{i\in [n]}w_i$ denote the total weight.
The Poissonian random graph or Norros-Reittu model \cite{NR06} is a multi-graph generated by creating  $X_{ij}$ many edges between vertices $i$ and $j$ independently, where 
    \begin{eq}
    X_{ij} \sim \mathrm{Poisson} (w_iw_j/\ell_n).
    \end{eq}
We denote this random multi-graph by $\NR$. 
The naturally associated simple graph, denoted by $\sNR$, is obtained by creating an edge between  $i$ and $j$ precisely when $X_{ij} \geq 1$. Thus an edge is kept between  vertex $i$ and $j$ independently with probability
	\begin{equation}
	\label{eq:p-ij-NR-defn}
	p_{ij} =1-\e^{-w_iw_j/\ell_n}.
	\end{equation}
The percolated graph $\rNR(\bw,\per)$ is obtained by keeping each edge of the multi-graph independently with probability $\per$. This deletion process is also independent of the randomization of the graph. If the underlying graph for percolation is the single-edge graph $\sNR$ instead, we denote the percolated graph by $\rsNR(\bw,\per)$.

In order to ensure that we are in the scale-free regime, we work with the following choice of vertex weights:
	\begin{assumption}[Scale-free weight structure]
	\label{assumption-NR}
	\normalfont
	For some $\tau \in (2,3)$, consider the distribution function~$F$ satisfying $[1-F](w) = Cw^{-(\tau-1)}$ for some $C>0$, and let $w_i = [1-F]^{-1}(i/n)$. 
	\end{assumption} 
	
In the above case, if $W_n$ denotes the weight of a vertex chosen uniformly at random, then $W_n$ will satisfy an asymptotic power-law in the sense that $\PR(W_n>w)=Cw^{-(\tau-1)},$ and, as a result, so will the degree distribution of the graph (see \cite[Chapter 6]{Hof17}), so that indeed we are dealing with a scale-free random graph. Further,
		\begin{equation}\label{ell-n}
		\E[W_n] = \frac{1}{n}	\sum_{i\in [n]}w_i \to \mu = \E[W],
		\end{equation}
and, for all $i\in[n]$,
	\begin{eq}\label{eq:weight-assumption}
w_i = [1-F]^{-1} \Big(\frac{i}{n}\Big) = \cf \Big(\frac{n}{i} \Big)^\alpha.
\end{eq}
where $\cf=C^{-1/(\tau-1)}>0$ with $C$ from Assumption \ref{assumption-NR}, and we recall that $\alpha = 1/(\tau-1)$ as in~\eqref{eqn:notation-const}. Throughout $\cf$ will denote the special constant appearing in \eqref{eq:weight-assumption}.

\subsection{Main results}
The percolation phase transition on $\NR$ model lies in the universality class of the configuration model studied in \cite{DHL19}. The critical window in this case consists of $\percn = \lambda n^{-\eta}$ for fixed $\lambda>0$ and $\eta$ as in \eqref{eqn:notation-const}. 
Thus, the supercritical behavior is observed for percolation probability is given by 
\begin{eq}\label{perc-prob}
\percn = \lambda_n n^{-\eta} = \lambda_n n^{-(3-\tau)/(\tau-1)} , \quad \text{where }\lambda_n \to\infty \text{ and } \lambda_n = o(n^{(3-\tau)/(\tau-1)}).
\end{eq}
The condition that $\lambda_n = o(n^{(3-\tau)/(\tau-1)})$ corresponds to $\percn=o(1)$, which we may assume without loss of generality, since there is a giant component whenever $\liminf_{n\rightarrow \infty}\percn>0$.
Let $(|\sCi(\per)|)_{i\geq 1}$ be the component sizes of $\rNR(\bw,\per)$, arranged in non-increasing order (breaking ties arbitrarily). 
The following theorem describes our main result for the multi-edge setting:
\begin{theorem}[Barely supercritical multi-edge setting]
\label{thm:asymp-multi}
Under Assumption~\ref{assumption-NR} and \eqref{perc-prob}, as $n\to\infty$, 
\begin{eq}
\frac{|\sC_{\sss (1)}(\percn)|}{n\percn^{1/(3-\tau)} } \pto \mu\kappa^{1/(3-\tau)}, 
\end{eq} 
where $\kappa = \cf^{\tau-2} \Gamma(3-\tau)$, and $\Gamma(\cdot)$ denotes the gamma function. Further, for any $j\geq 2$,   $|\sC_{\sss (j)}(\percn)| = \oP(n\percn^{1/(3-\tau)})$. 
\end{theorem}	
\begin{remark}[Critical scaling for $\rNR(\bw,\percn)$] \normalfont 
Since $\rNR(\bw)$ lies in the same universality class as the configuration model, one would expect that the components of $\rNR(\bw,\percn)$ exhibit critical behavior with identical scaling limits and scaling exponents as in \cite{DHL19} when $\lambda$ is fixed.
We do not study the critical behavior of percolation on  $\rNR(\bw)$ in this paper. 
\end{remark}

\begin{remark}[Comparison to other universality classes.] \normalfont It is worthwhile to compare the above barely super-critical behavior with that  for  other universality classes. The two well-known universality classes in the literature appear in the context of $\tau\in (3,4)$ and $\tau >4$, and the barely super-critical behavior in these universality classes was studied in \cite{HJL16,JL09}.  
It turns out that, for $\tau >4$ and $\tau\in (3,4)$, the size of the newly born giant grows as $n \percn$ and $n \percn^{1/(\tau-3)}$ respectively. On the other hand, Theorem~\ref{thm:asymp-multi} proves that the size of the newly born giant scales as $n\percn^{1/(3-\tau)}$ for $\tau \in (2,3)$. 
\end{remark}
The next theorem compares the giant components of $\rNR(\bw,\percn)$ and $\rsNR(\bw,\percn)$ for $\pi_n = \lambda_n n^{-\seta}$ and $\lambda_n\to\infty$. 
Let $(|\sCi^*(\per)|)_{i\geq 1}$ be the component sizes of $\rsNR(\bw,\per)$, arranged in non-increasing order (breaking ties arbitrarily).

\begin{theorem}[Barely supercritical single-edge setting]
\label{thm:asymp-single}
Under Assumption~\ref{assumption-NR}, for $\percn = \lambda_n n^{-\seta}$ with $\lambda_n \to \infty$ and $\lambda_n = o(n^{\seta})$,
\begin{eq}
\label{super-critical-equivalent}
\big||\sC_{\sss (1)}(\percn)|-|\sC^*_{\sss (1)}(\percn)|\big| = \oP(n\percn^{1/(3-\tau)}).
\end{eq}
Further, for any $j\geq 2$,   $|\sC_{\sss (j)}^*(\percn)| = \oP(n\percn^{1/(3-\tau)})$.
\end{theorem}

\paragraph*{Proof ideas.} 
We prove Theorem \ref{thm:asymp-multi} using an {\em exploration process} on $\rNR(\bw,\per)$. Here, it is very helpful that $\rNR(\bw,\per)$ has the same law as $\rNR(\bw\per)$. 
This allows us to use an exploration process suggested in \cite{BHL12}, which relies on creating Poisson$(\percn w_i/\ell_n)$ many potential neighbors for each visited vertex $i$ and deciding the status of the potential vertices sequentially by sampling marks from a size-biased distribution. This construction significantly simplifies the analysis of the exploration process (see Section~\ref{sec:exploration} for details).
Since the exploration process has a {\em only one} large excursion, there also is a unique largest connected component. Further, since the size of the largest excursion {\em concentrates}, also the largest connected component does so.

Theorem \ref{thm:asymp-single} is proved by using that an obvious coupling exists of $\rNR(\bw,\per)$ and $\rsNR(\bw,\per)$ that satisfies $|\sC_{\sss (1)}(\percn)|\geq |\sC^*_{\sss (1)}(\percn)|$. Thus, it suffices to show that 
    \[
    |\sC_{\sss (1)}^*(\percn)|\geq (1+\op(1))n\percn^{1/(3-\tau)}\mu\kappa^{1/(3-\tau)}.
    \]
For this, we will rely on the fact that the restriction of $\rsNR(\bw,\per)$ to the vertices in $[N_n]$, for an appropriate choice of $N_n$ (see \eqref{Nn-supercrit}), is a rank-1 inhomogeneous random graph. This turns out to be a rank-1 inhomogeneous random graph that is always supercritical, i.e., there is a `tiny giant' component inside the subgraph $[N_n]$ containing a constant times $N_n$ vertices. The main idea is to show that this tiny giant, together with all its {\em direct} neighbors, has size $(1+\op(1))n\percn^{1/(3-\tau)}\mu\kappa^{1/(3-\tau)}$. This completes the proof of Theorem \ref{thm:asymp-single}.

\paragraph*{Open problems.}
As remarked in the introduction, \cite{BDH18} shows that there is a giant in $\rsNR(\bw,\per)$ $\percn = \lambda n^{-\seta}$ with $\lambda \in ( \lambda_c,\infty)$ for some explicitly computable $\lambda_c>0$.
Theorem \ref{thm:asymp-single} shows that, when $\lambda = \lambda_n \to\infty$, the giant components inside $\rsNR(\bw,\per)$ and $\rNR(\bw,\per)$ have the same asymptotic sizes. 
However, if $\lambda\in ( \lambda_c,\infty)$ is fixed, then we believe that this is not the case anymore. 
More formally, we conjecture that
    \begin{eq}
    \label{critical-different}
    \frac{1}{\sqrt{n}}\big||\sC_{\sss (1)}(\percn)|-|\sC^*_{\sss (1)}(\percn)|\big| \pto \eta \in (0,\infty).
    \end{eq}
Intuitively, the methods for obtaining the size of the giant of $\rsNR(\bw,\per)$ explained above is conceptually similar to the $\lambda$ fixed case in \cite{BDH18}. 
In both cases, we identify a core of high-degree vertices that is approximately an inhomogeneous random graph, and the emergence of a `tiny giant' inside this core makes a giant appear in the whole graph. However, the key distinction in the $\lambda$ fixed case is that the graph inside the core is not a rank-one inhomogeneous random graph any more. This makes the limiting quantity in \cite{BDH18} difficult to compute explicitly. 
However, due to the non-rank-one structure of the core in \cite{BDH18}, we believe \eqref{critical-different} to be true. We leave this as an interesting open question.

Another related open problem concern the identification of the percolation phase transition for scale-free {\em erased} configuration models, where multi-edges are merged before applying percolation. The paper \cite{DHL19} studies the multi-edge setting, related to the present paper, both in the critical regime as well as in the barely supercritical regime, but the single-edge setting that corresponds to \cite{DHL19} is open. The challenge there lies in studying the emergence of the tiny giant inside high-degree vertices. 
One would a require suitable modification of \cite{BJR07} to incorporate the dependency of edge occupancy that occur in the erased configuration model.

\paragraph*{Organization. }
This paper is organised as follows. In Section \ref{sec-pf-thm1}, we prove Theorem \ref{thm:asymp-multi}, and in Section \ref{sec-pf-thm2}, we prove Theorem \ref{thm:asymp-single}.

\section{Super-critical phase for the multi-edge setting}
\label{sec-pf-thm1}
In this section, we complete the proof of Theorem~\ref{thm:asymp-multi}. We start by setting up some preliminary estimates and recalling a useful semi-martingale concentration inequality in Section~\ref{sec:preliminaries}. The main tool to prove Theorem~\ref{thm:asymp-multi} is the analysis of an exploration process that we explain in detail in Section~\ref{sec:exploration}. 
In Section~\ref{sec:large-early}, we prove that a large component is explored early in the exploration process. Finally, we complete the proof of Theorem~\ref{thm:asymp-multi} in Section~\ref{sec:proof-supercrit-NR}.
Throughout this section, we will assume that $\percn$ satisfies \eqref{perc-prob}.

\subsection{Preliminaries}\label{sec:preliminaries}
We start with the following estimate that will be used repeatedly in our proof:
\begin{lemma}[Asymptotics of Laplace-type transform] \label{lem:laplace}
Fix $t>0$ and let $t_n:= t\percn^{1/(3-\tau)}$ and $\beta_n = n\percn^{1/(3-\tau)}$ where $\percn = o(1)$. 
Then, 
\begin{eq}\label{laplace-two-forms}
\sum_{i\in [n]}  \frac{w_i}{\ell_n} \bigg( 1- \bigg(1-\frac{w_i}{\ell_n}\bigg)^{t \beta_n} \bigg) = \kappa (t_n/\mu)^{\tau-2} (1+o(1)),  
\end{eq}
where 
    \eqn{
    \label{kappa-def-NR-multiple-2}
   \kappa
    =\cf^{\tau-2}\Gamma(3-\tau),
    }
and $\Gamma(\cdot)$ is the gamma function. 
\end{lemma}
\begin{proof}
Using \eqref{eq:weight-assumption}, we simplify 
\begin{eq}
 \sum_{i\in [n]}  \frac{w_i}{\ell_n}\bigg( 1- \bigg(1-\frac{w_i}{\ell_n}\bigg)^{t \beta_n} \bigg) & \asymp  \frac{1}{\mu n} \sum_{i\in [n]} \cf \Big(\frac{n}{i}\Big)^{\alpha} \bigg(1- \bigg(1-\frac{\cf}{\mu n}\Big(\frac{n}{i}\Big)^\alpha\bigg)^{t\beta_n}\bigg)\\
& \asymp  \frac{\cf}{\mu} \int_0^1  x^{-\alpha} \bigg(1- \Big(1-\frac{\cf}{\mu n}x^{-\alpha}\Big)^{t\beta_n}\bigg) \dif x\\
&\asymp \frac{\cf}{\mu} 
    \int_0^1 x^{-\alpha}\big[1-\e^{-(\cf/\mu) x^{-\alpha}  t_n }\big]\dif x. 
    \end{eq}
Denote $z=(\cf/\mu) x^{-\alpha}  t_n$, so that $\dif x=- (\cf t_n/\mu )^{1/\alpha} z^{-1-1/\alpha}/\alpha \times \dif z.$ Then, by this change of variables,
    \eqan{
    \sum_{i\in [n]}  \frac{w_i}{\ell_n}\bigg( 1- \bigg(1-\frac{w_i}{\ell_n}\bigg)^{t \beta_n} \bigg)
    &\asymp \frac{1}{t_n \alpha}(\cf t_n/\mu)^{1/\alpha}
    \int_{\cf t_n/\mu}^{\infty}  z^{-1/\alpha}[1-\e^{-z}]\dif z \asymp \kappa \mu^{-(\tau-2)}t_n^{\tau-2},
    }
where
    \eqn{
    \label{kappa-def-1}
    \kappa=\frac{\cf^{\tau-1}(\tau-1)}{\mu}
    \int_0^{\infty}  z^{-(\tau-1)}[1-\e^{-z}]\dif z,
    }
and we have used the fact that $t_n\to 0$. The proof of \eqref{laplace-two-forms} now follows, with $\kappa$ as in \eqref{kappa-def-1} above. 
We close by showing the value of $\kappa$ indeed simplifies to \eqref{kappa-def-NR-multiple-2}. We first note that
    \eqn{\kappa=\frac{\cf^{\tau-1}(\tau-1)}{\mu}
    \int_0^{\infty}  z^{-(\tau-1)}[1-\e^{-z}]\dif u
    =\cf^{\tau-1}\frac{\tau-1}{(\tau-2)\mu} \Gamma(3-\tau).
    }
Further,
    \eqn{
    \label{mu-comp}
    \mu=\int_0^1 \cf u^{-\alpha}\dif u
    =\cf \frac{1}{1-\alpha}
    =\cf \frac{\tau-1}{\tau-2},
    }
so that the proof of \eqref{kappa-def-NR-multiple-2} now follows.
\end{proof}
Next, we state a semi-martingale inequality from  \cite[Lemma 2.2]{Jan94} that we will use extensively. 
Recall that a real-valued process $X$ defined on the filtered probability space $(\Omega,\cF,(\cF_t)_{t \geq  0},\PR)$ is called a semimartingale if it can be decomposed as $X_{t}=M_{t}+A_{t}$,
where $M$ is a local martingale and $A$ is a c\`adl\`ag adapted process of locally bounded variation. The process $A$ is also called the drift of $X$. The following holds for any bounded semi-martingale:
\begin{lemma}[{\cite[Lemma 2.2]{Jan94}}]\label{lem:semi-mart-conc}
If $(X(t))_{t\geq 0}$ is a bounded semimartingale with drift $(\zeta(t))_{t\geq 0}$, then 
\begin{eq}
\E \bigg[ \sup_{t\in [0,T]} |X(t)|^2\bigg] \leq 13 \E[|X(T)|^2] + 13T \int_0^T \E[\zeta(t)^2] \dif t. 
\end{eq}
\end{lemma}
If $X = (X_k)_{k\geq 1}$ is a discrete-time process, then it is a semimartingale with drift 
\begin{eq}\label{drift-defn}
A_k = \sum_{j=1}^{k} \E[X_j-X_{j-1} \vert \cF_{j-1}].
\end{eq}
We will use this expression later. 
\subsection{Scaling limit of the exploration process} \label{sec:exploration}
Define $\ww_i = \percn w_i$ and $\bell_n = \sum_{i\in [n]} \ww_i$. 
Let us start by noting that the distribution of $\rNR(\bw,\percn)$ is identical to  $\rNR(\bar{\bld{w}})$. For this reason, we will be analyzing $\rNR(\bar{\bld{w}})$ throughout this section. 
The analysis relies on a breadth-first exploration from \cite[Section 2.1]{BHL12}.

By a \emph{mark}, we mean a random variable $M$ having distribution 
\begin{eq}\label{eq:mark-distn}
\PR(M=i) = \frac{\ww_i}{\bell_n}=\frac{w_i}{\ell_n}, \quad \text{for }i\in [n].
\end{eq}
Note that each vertex $i$ has $\mathrm{Poisson}(\ww_i)$ many edges incident to it. 
The idea is to create a  \emph{potential neighbor} corresponding to each of these incident edges. 
One can generate the neighborhood of a vertex~$i$ by generating $\mathrm{Poisson}(\ww_i)$ many edges,  creating a potential vertex corresponding to each of these edges,  assigning i.i.d.\ marks to these potential vertices sampled from the distribution~\eqref{eq:mark-distn}, and finally identifying potential vertices with the same mark. This scheme suggests the following exploration on the graph: 

\begin{algo}[Exploration algorithm]\label{algo:exploration} \normalfont 
We first generate $(M_l)_{l\geq 1}$ from the mark distribution~\eqref{eq:mark-distn} in an i.i.d.\ manner. 
Think of $(M_l)_{l\geq 1}$ as corresponding to the {\em labels} of the explored vertices. The exploration process then proceeds as follows:
\begin{itemize}
    \item[(S0)] At step 1, we include $M_1$ in the explored set and create $\mathrm{Poisson}(\ww_{M_1})$ many potential neighbors. 
    
    \item[(S1)] At step $l$, if $M_l \in \{M_j\colon j\in[l-1]\}$, then no new vertex is found. 
    If there is a potential vertex, then we pick one of them and delete it. If there are no potential vertices then we do not do anything. Proceed to time $l+1$. 
    
    \item[(S2)]  If $M_l \notin \{M_j\colon j\in[l-1]\}$, then a new vertex is found at step $l$.
    If there is a potential vertex, then we pick one of them and assign it label $M_l$. Delete this vertex from the potential vertex set and add it to the explored set. If there is no potential vertex, then we add a new vertex labelled $M_l$ to the explored set. Finally, in either case, create $\mathrm{Poisson}(\ww_{M_l})$ many new potential neighbors.
\end{itemize}
\end{algo}
Consider the following walk associated to the above exploration process: $Z_n(0) = 0$, and 
\begin{eq} \label{eq:exploration-potential}
Z_n(l) = Z_n({l-1}) + X_n(l) - 1, 
\end{eq}
where the conditional distribution of $X_n(l)$, conditionally on $\{M_1,\dots, M_l\}$ and $\{X_n(1),\dots, X_n(l-1)\}$, is given by a $ \mathrm{Poisson}(\ww_{M_l}) \ind{M_l \notin \{M_1,\dots,M_{l-1}\}}$. 
We can think of $X_n(l)$ as the number of potential neighbors added at time $l$. A component of size at least $2$ is explored between times $l_1$ and $l_2$ ($l_2 > l_1+1$) if $\min_{u\leq l_1} Z_n(u) = \min_{u\leq l_2} Z_n(u) - 1$.

Throughout, we let $\cF_n(l)$ denote the sigma-algebra with respect to which $\{M_1,\dots, M_l\}$ and $\{X_n(1),\dots, X_n(l)\}$ are measurable. Let $V_l = \{M_1,\dots, M_l\}$ be the set of marks discovered up to time $l$. 
We can simplify the above walk as 
\begin{eq}
Z_n(l) = Z_n(0) + \sum_{i\in [n]} \xi_n(i) \ind{i\in V_l} -l, 
\end{eq}
where $(\xi_n(i))_{i\in [n]}$ is a sequence of independent random variables with $\xi_n(i) \sim \mathrm{Poisson}(\bar w_i)$. Also, $(\xi_n(i))_{i\in [n]}$ is independent of $(M_l)_{l\geq 1}$. 
Let $\beta_n = n\percn^{1/(3-\tau)}$, and define the rescaled walk by 
\begin{eq}
\bar{Z}_n(t) = \beta_n^{-1} Z_n(\lfloor t \beta_n\rfloor).   
\end{eq}

Our main result concerning the process limit of the exploration process is the following theorem:

\begin{theorem}[Exploration process limit]\label{thm:conv-expl-proc} Fix any $T>0$. 
Under Assumption~\ref{assumption-NR} and \eqref{perc-prob},  as $n\to \infty$, 
\begin{eq}
\sup_{t \in [0, T]} \big|\bar{Z_n} (t) - z(t) \big| \pto 0,
\end{eq}where $z(t) = \mu^{3-\tau} \kappa t^{\tau-2} - t$ and  $\kappa = \cf^{\tau-2} \Gamma(3-\tau)$. 
\end{theorem}
Let $\bld{S}_n$ be the walk defined as 
\begin{eq}
S_n(l) =  Z_n(0) + \sum_{i\in [n]} \ww_i \ind{i\in V_l} -l. 
\end{eq}
and define $\bar{S}_n(t) = \beta_n^{-1} S_n(\lfloor t\beta_n\rfloor).$ 
Henceforth, we will assume Assumption~\ref{assumption-NR} and \eqref{perc-prob} hold in this section.
We first prove the following 
proposition: 
\begin{proposition}[Modified exploration process limit]\label{prop:S-limit}
For any fixed $T>0$, as $n\to \infty$, 
\begin{eq}
\sup_{t \in [0, T]} \big|\bar{S}_n (t) - z(t) \big| \pto 0,
\end{eq}where $z(t) = \mu^{3-\tau} \kappa t^{\tau-2} - t$ and  $\kappa = \cf^{\tau-2} \Gamma(3-\tau)$. 
\end{proposition}

\begin{proof}
First, note that, by \eqref{eq:mark-distn}, 
\begin{eq}\label{eq:prob-i-in-V}
\PR(i\in V_l) = 1- \PR(i\notin V_l)  =1- \bigg(1-\frac{w_i}{\ell_n} \bigg)^l. 
\end{eq}
Therefore, taking $l = \lfloor t \beta_n \rfloor $, we get
\begin{eq}\label{eq:expt-Z-n-1}
\E[\bar{S}_n(t)] &\asymp \frac{1}{\beta_n} \sum_{i\in [n]}  \ww_i\bigg(1- \bigg(1-\frac{w_i}{\ell_n}\bigg)^{t \beta_n} \bigg)- t.
\end{eq}
Applying Lemma~\ref{lem:laplace} yields
\begin{eq}\label{expt-S-comp}
\E[\bar{S}_n(t)]& \asymp  \frac{\ell_n \percn }{n \percn^{1/(3-\tau)}} \sum_{i\in [n]}  \frac{w_i}{\ell_n} \bigg(1- \bigg(1-\frac{w_i}{\ell_n}\bigg)^{t \beta_n} \bigg) - t\\
&\asymp\frac{ \mu \percn }{\percn^{1/(3-\tau)}}\times  \kappa \big(t \percn^{1/(3-\tau)}/\mu\big) ^{\tau-2} - t\\
& \asymp\frac{ \mu^{3-\tau} \percn }{\percn^{1/(3-\tau)}}\times  \kappa \big(t \percn^{1/(3-\tau)} \big) ^{\tau-2} - t\\
& \asymp  \kappa t^{\tau-2} \mu^{3-\tau} - t=z(t),
\end{eq}
and all the above approximations hold uniformly over $t\leq T$, where $T>0$ is fixed. 
The proof follows if we can show that 
\begin{eq}\label{expt-sup-S}
\lim_{n\to\infty}\E \bigg[\sup_{t\in [0,T]} \big|\bar{S}_n(t) - \E[\bar{S}_n(t)]\big|^2\bigg] = 0. 
\end{eq}
We will apply Lemma~\ref{lem:semi-mart-conc} to the semi-martingale $(\bar{S}_n(t) - \E[\bar{S}_n(t)])_{t\geq 0}$ whose drift can be computed by \eqref{drift-defn}. 
Indeed, 
\begin{eq}\label{var-S-n}
\E[|S_n(l) - \E[S_n(l)]|^2] = \sum_{i\in [n]} \ww_i^2 \Var ( \ind{i\in V_l}) + \sum_{i,j\in [n]: i\neq j} \ww_i \ww_j \Cov(\ind{i\in V_l}, \ind{j\in V_l}).  
\end{eq}
If $0< a,b<x$, then $(1-a/x) (1-b/x) \geq 1- (a+b)/x$. Therefore, for $i\neq j$,
\begin{eq}\label{negative-correlation}
\PR(i,j \notin V_l) = \bigg(1-\frac{w_i+w_j}{\ell_n} \bigg)^l \leq \bigg(1-\frac{w_i}{\ell_n} \bigg)^l\bigg(1-\frac{w_j}{\ell_n} \bigg)^l = \PR(i \notin V_l)\PR(j \notin V_l),
\end{eq}
and hence the random variables $\ind{i\in V_l}$ and $\ind{j\in V_l}$ are negatively correlated. 
Thus, we can drop the covariance term in \eqref{var-S-n} to arrive at an upper bound. 
Using \eqref{eq:prob-i-in-V}, we conclude that
\begin{eq}\label{eq:var-bound-a}
\E[|\bar{S}_n(t) - \E[\bar{S}_n(t)]|^2] &\leq  \beta_n^{-2} \sum_{i\in [n]} \ww_i^2 \PR(i\in V_{\lfloor t \beta_n \rfloor}) \\
&\leq \frac{\percn^2}{\beta_n^2} \sum_{i\in [n]} w_i^2\bigg(1- \bigg(1-\frac{w_i}{\ell_n}\bigg)^{t \beta_n} \bigg) \\
& \leq \frac{w_1 \percn }{\beta_n} \times \frac{\percn}{\beta_n} \sum_{i\in [n]} w_i\bigg(1- \bigg(1-\frac{w_i}{\ell_n}\bigg)^{t \beta_n} \bigg).
\end{eq}
By the computations in \eqref{expt-S-comp}, the second factor is $O(1)$. By \eqref{perc-prob}, the first factor is 
$$
\frac{w_1 \percn }{\beta_n} \leq C \frac{n^{\alpha} \percn}{n \percn^{1/(3-\tau)} }\leq  
C\big(n^\alpha \lambda_n^{1/(3-\tau)} n^{-1/(\tau-1)} \big)^{-(\tau-2)} = C \lambda_n^{-(\tau-2)/(3-\tau)}  \to 0,$$
since 
$\lambda_n \to \infty$.
Thus, for any fixed $t>0$, 
\begin{eq}\label{eq:expt-particular}
\E[|\bar{S}_n(t) - \E[\bar{S}_n(t)]|^2] \to 0. 
\end{eq}
Next we consider the drift of $(\bar{S}_n(t)  )_{t\geq 0}$. Indeed, 
\begin{eq}
\big( \ind{i\in V_l} - \PR(i\in V_l \mid \cF_n(l-1)))_{l\geq 0} \text{ is a martingale with respect to } (\cF_n(l))_{l\geq 0}. 
\end{eq}
Therefore, the drift of $(S_n(l) - \E[S_n(l)] )_{l\geq 0}$ at time $l$ is given by 
\begin{eq}
\zeta_n(l):=\sum_{i\in [n]} \ww_i \big( \PR(i\in V_l \mid \cF_n(l-1)) - \PR(i\in V_l)\big).
\end{eq}
Now,
\begin{eq}\label{drift-expression}
\PR(i\in V_l \mid \cF_n(l-1)) = \ind{i\in V_{l-1}} + \frac{w_i}{\ell_n} \ind{i\notin V_{l-1}}=\frac{w_i}{\ell_n}+ \Big(1-\frac{w_i}{\ell_n}\Big) \ind{i\in V_{l-1}}. 
\end{eq}
Using the negative correlation again between $\ind{i\in V_l}$ and $\ind{j\in V_l}$ for $i\neq j$, 
\begin{eq}\label{square-zeta}
\E[\zeta_n(l)^2] &\leq  \sum_{i\in [n]} \ww_i^2 \Var (\PR(i\in V_l \mid \cF_n(l-1))) \\
&= \sum_{i\in [n]} \ww_i^2\Big(1-\frac{w_i}{\ell_n}\Big)^2  \Var (\ind{i\in V_{l-1}}) \leq \sum_{i\in [n]} \ww_i^2  \PR(i\in V_{l-1}). 
\end{eq}
Using the computation in \eqref{eq:var-bound-a}, we get 
\begin{eq}\label{drift-small}
\frac{1}{\beta_n^2}\E[\zeta_n(\lfloor t\beta_n\rfloor)^2] \to 0,
\end{eq}
uniformly over $t\in [0,T]$, for any fixed $T>0$. 
Finally, an application of Lemma~\ref{lem:semi-mart-conc} together with \eqref{eq:expt-particular} and \eqref{drift-small} completes the proof of \eqref{expt-sup-S}. 
Thus, the proof of Proposition~\ref{prop:S-limit} follows by combining \eqref{expt-S-comp} and \eqref{expt-sup-S}. 
\end{proof}
Next we show that the two processes $\bld{Z}_n$ and $\bld{S}_n$ are uniformly close: 
\begin{proposition}[Closeness to modified process]\label{prop-diff}
For any fixed $T>0$, as $n\to \infty$, 
\begin{eq}
\sup_{t \in [0, T]} \big|\bar{Z}_n (t) - \bar{S}_n(t)\big| \pto 0. 
\end{eq}
\end{proposition}
\begin{proof} It suffices to prove that 
\begin{eq}
\lim_{n\to\infty}\E \bigg[\sup_{t\in [0,T]} \big|\bar{Z}_n(t) - \bar{S}_n(t)\big|^2\bigg] = 0,
\end{eq}
for which we will again use Lemma~\ref{lem:semi-mart-conc}. 
Indeed, 
\begin{eq}\label{diff-expt}
\E[|Z_n(l) - S_n(l)|^2] = \sum_{i\in [n]} \E[(\xi_n(i) - \ww_i)^2 \ind{i\in V_l}]  = \sum_{i\in [n]} \ww_i \PR(i\in V_l) = O(\beta_n),
\end{eq}
where the last step follows using \eqref{expt-S-comp}. 
Using the independence of $(\xi_n(i))_{i\in [n]}$ and $(M_l)_{l\geq 1}$, we see that 
\begin{eq}
\big((\xi_n(i) - \ww_i) \ind{i\in V_l}\big)_{l\geq 0}\text{ is a martingale with respect to } (\cF_n(l))_{l\geq 0},
\end{eq}
and thus the drift of $(Z_n(l) - S_n(l))_{l\geq 0}$ is zero at all the times. Applying Lemma~\ref{lem:semi-mart-conc}, together with \eqref{diff-expt}, the proof of Proposition~\ref{prop-diff} follows. 
\end{proof}
\medskip

\begin{proof}[Proof of Theorem~\ref{thm:conv-expl-proc}] The proof follows immediately from Propositions~\ref{prop:S-limit}~and~\ref{prop-diff}.
\end{proof}

\paragraph*{Dealing with repeat vertices.} We end this section by analyzing how many times we encounter repeated marks during the exploration process. We call an occurrence of a repeat of a mark a {\em repeat vertex}.
Recall that the exploration process in \eqref{eq:exploration-potential} removes a \emph{potential} vertex at each step, and thus the excursion length of the process in \eqref{eq:exploration-potential} counts the total number of potential vertices encountered during the exploration of components, rather than the component sizes. 
In order to know the component sizes from the excursion lengths of the exploration process, we need to show that most of the potential vertices become vertices of the graph during the exploration, and there are negligibly many repeat vertices. 
We prove the following: 
\begin{proposition}
[Few repeat vertices]
\label{prop:real-vertices}
Let $R(k) = \sum_{l=1}^k \ind{M_l \in \{M_1,\dots,M_{l-1}\}}$ denote the number of repeat vertices discovered until time $k$. Then, for any fixed $t>0$, $\beta_n^{-1}R(\floor{t\beta_n}) \xrightarrow{\sss \PR} 0$. 
\end{proposition}
\begin{proof}
Note that 
\begin{eq}
\E[R(k)] &= \E\bigg[\sum_{l = 1}^k \PR(M_l \in \{M_1,\dots,M_{l-1}\} \mid \sF_n(l-1))\bigg] \\
&= \E\bigg[\sum_{l = 1}^k \frac{1}{\bell_n} \sum_{j \in [n] } \ww_j \ind{j\in V_{l-1}}\bigg] \leq \frac{k}{\bell_n} \times \E\bigg[  \sum_{j \in [n] } \ww_j\ind{j\in V_{k-1}}\bigg] .
\end{eq}
Take $k = \floor{t\beta_n}$. 
Recall from the computations in \eqref{expt-S-comp} that the second factor is $O(\beta_n)$, and therefore, by Markov's inequality, $\sum_{j \in [n] } \ww_j\ind{j\in V_{k-1}} =\OP(\beta_n)$.
Thus, 
\begin{eq}
\frac{R(\floor{t\beta_n})}{\beta_n} \leq \frac{t\beta_n}{ \beta_n \bell_n} \times \OP(\beta_n) = \OP(\percn^{(\tau-2)/(3-\tau)}) \pto 0,
\end{eq}
and the proof follows. 
\end{proof}

\subsection{Large components are explored early} \label{sec:large-early}
In this section, we prove that the exploration process is likely to explore all the \emph{large} components within a time of order $\beta_n$. Our first lemma describes how the exploration of vertices  effect the expected forward degrees of the newly explored vertices:

\begin{lemma}[Expected forward degree during exploration]\label{lem:nu-n}
Let  $\nu_n(t) = \sum_{i\notin V_{\floor{t\beta_n}}} \ww_i^2/\sum_{i\notin V_{\floor{t\beta_n}}} \ww_i $. Then there exists a $t>0$ (sufficiently large) such that $\nu_n(t)<1/2$ with high probability. 
\end{lemma}
\begin{proof}
Let $W_n$ denote the weight of a vertex chosen uniformly at random. Then $W_n\xrightarrow{\sss d} W$ by Assumption~\ref{assumption-NR}, where $W$ is distributed as $F$. By \eqref{ell-n}, it follows that $(W_n)_{n\geq 1}$ is uniformly integrable. 
Since $|V_{\floor{t\beta_n}}| \leq t\beta_n = o(n)$, we conclude that $\sum_{i\notin V_{\floor{t\beta_n}}} \ww_i = \bell_n(1+\oP(1))$ for any $t>0$. 
Thus, it suffices to show that for a large enough $t$,   
\begin{eq}\label{nu-n-simpler-a}
\nu_n'(t) <\frac{1}{3} \quad\text{with high probability,}
\end{eq}
where $\nu_n'(t):=\frac{1}{\bell_n}\sum_{i\notin V_{\floor{t\beta_n}}} \ww_i^2$. 
Let us first bound $\E[\nu_n'(t)]$. Indeed, using $1-x \leq \e^{-x}$ for all $x>0$,  
\begin{eq}
\E[\nu_n'(t)] &= \frac{1}{\bell_n} \sum_{i\in [n]} \ww_i^2 \PR(i\notin V_{\floor{t\beta_n}} ) = \frac{\percn}{\ell_n} \sum_{i\in [n]} w_i^2 \Big(1-\frac{w_i}{\ell_n}\Big)^{t\beta_n} \leq \frac{\percn}{\ell_n} \sum_{i\in [n]} w_i^2 \e^{-\frac{w_i}{\ell_n}{t\beta_n}} 
\\
&\asymp \frac{\percn}{\mu} \times \frac{1}{n}   \sum_{i\in [n]} \cf^2 \Big(\frac{n}{i}\Big)^{2\alpha}  \e^{- \frac{\cf}{\mu n}(\frac{n}{i})^\alpha{t\beta_n}}\\
& \asymp \frac{\cf^2 \percn}{\mu}
    \int_0^1 x^{-2\alpha}\e^{-(\cf/\mu) x^{-\alpha}  t_n }\dif x.
\end{eq}
Denote $z=(\cf/\mu) x^{-\alpha}  t_n$, so that $\dif x=- (\cf t_n/\mu )^{1/\alpha} z^{-1-1/\alpha}/\alpha \times \dif z.$ Then, by this change of variables,
\begin{eq}\label{expt-nu-n-t}
\E[\nu_n'(t)] 
    &\leq  \frac{1}{\alpha}\mu^{1-1/\alpha} \cf^{1/\alpha}  \percn t_n^{-2+1/\alpha}
    \int_{\cf t_n/\mu}^{\infty}  z^{1-1/\alpha}\e^{-z}\dif z \\
    & \asymp \frac{1}{\alpha}\mu^{1-1/\alpha} \cf^{1/\alpha} \Gamma(3-\tau) \times t^{-(3-\tau)}.
\end{eq}
Thus, we can choose $t$ large enough so that $\E[\nu_n'(t)] <1/4$. Next we bound the variance of~$\nu_n'(t)$. 
To that end, we compute 
\begin{eq}
&\Var(\nu_n'(t) ) \\
&= \frac{1}{\bell_n^2}\bigg[ \sum_{i\in [n]} \ww_i^4 \Var ( \mathbbm{1}_{\{i\notin V_{\floor{t\beta_n}}\}}) + \sum_{i,j\in [n]\colon i\neq j} \ww_i^2\ww_j^2 \Cov(\mathbbm{1}_{\{i\notin V_{\floor{t\beta_n}}\}}, \mathbbm{1}_{\{j\notin V_{\floor{t\beta_n}}\}}\bigg].  
\end{eq}
Now, for $i\neq j$, we can again use the negative correlation between  $\mathbbm{1}_{\{i\notin V_{\floor{t\beta_n}}\}}$ and $\mathbbm{1}_{\{j\notin V_{\floor{t\beta_n}}\}}$ from \eqref{negative-correlation} to conclude that 
\begin{eq}\label{eq:var-bound}
\Var(\nu_n'(t) ) &\leq  \frac{1}{\bell_n^2} \sum_{i\in [n]} \ww_i^4 \Var ( \mathbbm{1}_{\{i\notin V_{\floor{t\beta_n}}\}}) \leq \frac{\percn^2}{\ell_n^2} \sum_{i\in [n]} w_i^4\bigg(1-\frac{w_i}{\ell_n}\bigg)^{t \beta_n} .
\end{eq}
Repeating the computation in \eqref{expt-nu-n-t}, we get, with $t_n=t\percn^{1/(3-\tau)}$,
\begin{eq}
\frac{1}{\ell_n}\sum_{i\in [n]} w_i^4\bigg(1-\frac{w_i}{\ell_n}\bigg)^{t \beta_n} \leq C t_n^{-4+1/\alpha}. 
\end{eq}
Thus, for any $t>0$, 
\begin{eq}\label{expt-dif-nu}
\Var(\nu_n'(t) ) &\leq C \frac{\percn^2}{\ell_n} (t \percn ^{1/(3-\tau)}) ^{-5+\tau} \leq \frac{Ct^{-5+\tau}}{\ell_n} \percn^{-(\tau-1)/(3-\tau)} \\
&\leq (1+o(1)) \frac{1}{\mu \lambda_n^{(\tau-1)/(3-\tau)}} \to 0. 
\end{eq}
Using Chebyshev's inequality, \eqref{nu-n-simpler-a} follows and thus proof of Lemma \ref{lem:nu-n} is complete. 
\end{proof}
We are now ready to show that all the large components are explored early during the exploration process. 
Let $\cG_n(t)$ be the sub-graph of $\rNR(\bar{\bld{w}})$ obtained by removing the vertices in $V_{\lfloor t\beta_n \rfloor}$. Then, conditionally on $V_{\lfloor t\beta_n \rfloor}$, $\cG_n(t)$ is again distributed as a Norros-Reittu model. Let $\sCi^{t}$ denote the $i$-th largest connected component in $\cG_n(t)$. The next proposition shows that $\sC_{\sss (1)}^{t}$ is microscopic when $t$ is large:

\begin{proposition}[Small components left afterwards]\label{prop:large-early}
There exists a large enough $t>0$ such that for any $\varepsilon>0$ 
\begin{eq}
\lim_{n\to\infty} \PR(|\sC_{\sss (1)}^{t}| > \varepsilon \beta_n ) = 0. 
\end{eq}
\end{proposition}
\begin{proof} 
By Lemma~\ref{lem:nu-n}, we can choose a large enough $t$ such that $\nu_n(t) < \tfrac{1}{2}$ with high probability.
Let $\tPR$ and $\tE$ denote the conditional probability and expectation respectively conditionally on  $V_{\lfloor t\beta_n \rfloor}$.
For this choice of $t$, it suffices to show that 
\begin{eq}\label{eq:sum-square}
\tE\bigg[\sum_{i=1}^\infty |\sCi^{t}|^2 \mathbbm{1}_{\{|\sCi^{t}| \geq 2\}} \bigg] =o(\beta_n^2), 
\end{eq}almost surely. 
We first compute $|\sC^{t} (v)|$ for all $v\in V_{\lfloor t\beta_n \rfloor}^c$ that are not isolated in $\cG_n(t)$, i.e., $|\sC^{t} (v)| \geq 2$. 
Indeed, let $\cT_n(v,t)$ be the multi-type branching process where we start with vertex $v\in V_{\lfloor t\beta_n \rfloor}^c$, and at each generation, a vertex with mark $i$ ($i\in V_{\lfloor t\beta_n \rfloor}^c$) produces a $\mathrm{Poisson}(\ww_i)$ many offspring and each of these off-springs are assigned mark $j$ ($j\in V_{\lfloor t\beta_n \rfloor}^c$) with probability $w_j/\sum_{k\in V_{\lfloor t\beta_n \rfloor}^c} w_k$. Then the criticality parameter of this branching process is $\nu_n(t)<\tfrac{1}{2}$. 
Thus, by \cite[Lemma 2.3 (c)]{BHL12}, 
\begin{eq}
\tE[|\cT_n(v,t)|] \leq 1+ \frac{\ww_v}{1-\nu_n(t)},
\end{eq}almost surely. 
Next, note that if we discard the repeated marks in $\cT_n(v,t)$ as in Algorithm~\ref{algo:exploration}, then we get the breadth-first exploration tree of $\sC^{t} (v)$.
This gives a coupling between $\sC^{t} (v)$ and $\cT_n(v,t)$. Under this coupling, $|\sC^{t} (v)| \leq |\cT_n(v,t)|$ and the event $\{\cT_n(v,t) <2 \text{ and } |\sC^{t} (v)| \geq 2\} $ has probability zero. 
Therefore, 
\begin{eq}
\tE\Big[|\sC^{t} (v)| \mathbbm{1}_{\{|\sC^{t} (v)| \geq 2\}}\Big]&\leq \tE\Big[|\cT_n(v,t) | \mathbbm{1}_{\{|\cT_n(v,t) | \geq 2\}}\Big]\\
& = \tE[|\cT_n(v,t) | ] - \tE\Big[|\cT_n(v,t) | \mathbbm{1}_{\{|\cT_n(v,t) | =1\}}\Big]\\
& \leq 1+ \frac{\ww_v}{1-\nu_n(t)} - \tPR(|\cT_n(v,t) | =1) \\
& =\frac{\ww_v}{1-\nu_n(t)}+\tPR(|\cT_n(v,t) |\geq 2) \\
& \leq \ww_v+ \frac{\ww_v}{1-\nu_n(t)} \leq 3\ww_v,
\end{eq}
where in the one-but-last step we have used $\PR(|\cT_n(v,t) | \geq 2) = 1-\e^{-\ww_v} \leq \ww_v$. 
Therefore, 
\begin{eq}
\tE\bigg[\sum_{i=1}^\infty |\sCi^{t}|^2 \mathbbm{1}_{\{|\sCi^{t}| \geq 2\}}\bigg] =  \tE\bigg[\sum_{v \notin V_{\lfloor t\beta_n \rfloor} } |\sC^{t} (v)| \mathbbm{1}_{\{|\sC^{t} (v)| \geq 2\}} \bigg] \leq 3 \sum_{v \notin V_{\lfloor t\beta_n \rfloor} } \ww_v = O(n\percn) = o(\beta_n^2), 
\end{eq}
since $n\percn/\beta_n^2 = \percn ^{-(\tau-1)/(3-\tau)} n^{-1} = \lambda_n^{-(\tau-1)/(3-\tau)} = o(1)$. Thus, \eqref{eq:sum-square} follows and the proof of Proposition~\ref{prop:large-early} is complete.
\end{proof}

\subsection{Proof of Theorem~\ref{thm:asymp-multi}} \label{sec:proof-supercrit-NR}
With the above ingredients, we can complete the proof of Theorem~\ref{thm:asymp-multi}.
Let $\zeta = \mu \kappa^{1/(3-\tau)}$ be such that $z(t) = \kappa \mu^{3-\tau}t^{\tau-2} - t = 0$. 
Fix $\delta>0$ sufficiently small. 
Using Theorem~\ref{thm:conv-expl-proc}, for any fixed $T>0$ and  with high probability 
\begin{eq} \label{eq:delta-close}
\sup_{t \in [0, T]} \big|\bar{Z_n} (t) - z(t) \big| \leq \delta. 
\end{eq}
Then, there exists $\varepsilon = \varepsilon(\delta)>0$ ($\varepsilon\to 0$ as $\delta \to 0$) such that with high probability 
\begin{eq}
Z_n(l)>0 \geq \min_{u\leq l} Z_n(u) \quad \text{ for all }\quad l\in [\varepsilon \beta_n , (\zeta - \varepsilon) \beta_n]. 
\end{eq}
Therefore, the exploration process keeps on exploring a single component throughout the time interval $[\varepsilon \beta_n , (\zeta - \varepsilon) \beta_n]$.
Let $\sC'$ be the component that is being explored in the time interval $[\varepsilon \beta_n , (\zeta - \varepsilon) \beta_n]$. 
Recall that the excursion length of $\bld{Z}_n$ denotes the number of {\em potential} vertices in a component, which is approximately the component size due to Proposition~\ref{prop:real-vertices}.
Thus, Proposition~\ref{prop:real-vertices} and the above argument together  show that $|\sC'| \geq (\zeta - 2\varepsilon) \beta_n$ with high probability. 

Next, there exists 
$\varepsilon' = \varepsilon'(\delta)>0$ (where $\varepsilon'\to 0$ as $\delta \to 0$) such that $z(t+\varepsilon')<-3\delta$. 
By \eqref{eq:delta-close}, there exists $l\in [(\zeta - \varepsilon) \beta_n, (\zeta + \varepsilon') \beta_n]$ such that $Z_n(l) < \min_{j\leq l} Z_n(j)$, and hence a new component is finished exploring in $[(\zeta - \varepsilon) \beta_n, (\zeta + \varepsilon') \beta_n]$. Therefore, $\sC'$ is finished exploring before $(\zeta + \varepsilon') \beta_n$ and thus $|\sC'| \leq (\zeta + \varepsilon') \beta_n$. Taking $\delta>0$ to be arbitrarily small, we conclude that 
\begin{eq}
\beta_n^{-1}|\sC'| \pto \zeta. 
\end{eq}
In order to prove that $\sC' = \sC_{\sss (1)} (\percn)$, it remains to show that all other components have size $\oP(\beta_n)$. Fix $t$ large enough such that Proposition~\ref{prop:large-early} holds. Since $\bld{z}$ has only one excursion of positive length, all the components explored {\em before} time $t\beta_n$ that are different from $\sC'$ must have size $\oP(\beta_n)$. 
Further, Proposition~\ref{prop:large-early} shows that all the components explored {\em after} time $t\beta_n$ also have size $\oP(\beta_n)$. 
Thus, $\sC' = \sC_{\sss (1)} (\percn)$ and $|\sC_{\sss (2)} (\percn)| = \oP(\beta_n)$, and the proof of Theorem~\ref{thm:asymp-multi} is complete. \hfill$\blacksquare$

\section{Super-critical phase for single-edge setting}
\label{sec-pf-thm2}
In this section, we will prove Theorem~\ref{thm:asymp-single} regarding percolation on the single-edge model $\rsNR(\bw,\percn)$.
For this, we assume that 
\begin{eq}\label{percn-supercrit-single}
\percn = \lambda_n n^{-\seta} \quad \text{ with } \lambda_n \to \infty \text{ and } \lambda_n = o(n^{\seta}). 
\end{eq}
Note that an edge is kept in $\rsNR(\bw,\percn)$ between vertex $i$ and $j$ independently with probability
	\begin{equation}
	\label{eq:p-ij-NR-defn-2}
	\percn p_{ij} =\percn \big(1-\e^{-w_iw_j/\ell_n}\big).
	\end{equation}
Our strategy to prove \eqref{super-critical-equivalent} will be to obtain an asymptotic lower bound on the giant-size $|\sC_{\sss (1)}^*(\percn)|$ of $\rsNR(\bw,\percn)$ that matches with Theorem~\ref{thm:asymp-multi}.

First, we will show that the order of the giant component will be equal to the number of neighbors of a set of size $\Theta(N_n)$, where $N_n$ is such that $1/N_n$ is of the same order as $\percn w_{N_n}^2/\ell_n$. We present an overview of this argument in Section \ref{sec-overview-core}, in which we also explain how the remainder of the proof is organised.

\subsection{Formation of a giant core inside hubs: overview}
\label{sec-overview-core}
The order of the giant component will be the neighbors of a set of size $aN_n$, where $N_n$ satisfies that $1/N_n$ is of the same order as $w_{N_n}^2/\ell_n$. For such values of $N_n$ and all $u,v>0$,
    \eqn{
    \label{pij-intuition-Nn}
    N_np_{\lceil uN_n\rceil,\lceil uN_n\rceil}
    =N_n[1-\e^{-w_{\lceil uN_n\rceil}w_{\lceil uN_n\rceil}/\ell_n}]\rightarrow \frac{\cf^2}{\mu} (uv)^{-\alpha},
    }
so that $\rsNR(\bw,\percn)$ restricted to $[N_n]$ is a rank-one inhomogeneous random graph in the spirit of \cite{BJR07}.
Working this relation out yields that 
    \eqn{
    \label{Nn-supercrit}
    N_n=\lambda_n^{(\tau-1)/(3-\tau)} n^{(3-\tau)/2}.
    }
We will later use that \eqref{Nn-supercrit} is equivalent to
    \eqn{
    \label{Nn-supercrit-rew}
    \pi_n=\Big(\frac{N_n}{n}\Big)^{(3-\tau)/(\tau-1)}=\Big(\frac{n}{N_n}\Big)^{1-2\alpha},
    \quad \text{or}
    \quad
    N_n=n \percn^{(\tau-1)/(3-\tau)}.
    }
We now proceed to set up the main conceptual ingredients for the identification of the giant.
Fix a parameter $a>0$, and define 
	\eqn{
	\label{Nna-prime-def}
	N_n(a)=\lfloor a N_n \rfloor.
	}
Note that,  for $i\in \lceil N_n u \rceil $	and $u \in (0,a]$,
\begin{eq}\label{eq:asymp-w-N}
	w_{\lceil N_n u \rceil}
	=
	\cf u^{-\alpha}\Big(\frac{n}{N_n}\Big)^{\alpha}
	=\cf u^{-\alpha}\big(\lambda_n^{-(\tau-1)/(3-\tau)}n^{(\tau-1)/2}\big)^{\alpha} \asymp \sqrt{n} \lambda_n^{-1/(3-\tau)} \cf u^{-\alpha},
\end{eq}
and thus  $[N_n(a)]$ consists of vertices with weight at least of order $\sqrt{n}\lambda_n^{-1/(3-\tau)} a^{-\alpha}$. 
Consider the subgraph of $\rsNR(\bw,\percn)$ induced on $[N_n(a)]$. 
We denote this subgraph on $[N_n(a)]$ by $\cG_{\sss N_n(a)}$. 

The expression in \eqref{pij-intuition-Nn} suggests that $\cG_{\sss N_n(a)}$ is distributed approximately as a rank-1 inhomogeneous random graph. 
Also, $\cG_{\sss N_n(a)}$ is 
{\em sparse} in the sense that the number of edges grows linearly in the number of vertices in the graph.
Thus, the emergence of the giant component within $\cG_{\sss N_n(a)}$ can be studied using the general setting of inhomogeneous random graphs developed by Bollob\'as, Janson and Riordan in \cite{BJR07}. 
In particular, the results of 
\cite{BJR07} give that, for the choices in \eqref{Nn-supercrit}~and~ \eqref{Nn-supercrit-rew} 
and by \eqref{pij-intuition-Nn},
a unique and highly-concentrated giant exists inside $[N_n(a)]$ for {\em every} $a>0$.

In Section~\ref{sec:size-core-giant}, we make the connection with the key results from \cite{BJR07} explicit and state the relevant results for our proof. While the giant exists for every $a>0$, it does grow when $a$ grows, and thus we will need to take the limit as $a\rightarrow \infty$. 
The rest of Section~\ref{sec-pf-thm2} is devoted to the analysis of the limiting quantities as $a\to \infty$. 
In Section \ref{sec-lower-bound-single-edge}, we prove a {\em lower bound} on $|\sC_{\sss (1)}^*|$ by analyzing the size of the one-neighborhood of the giant of $\cG_{\sss N_n(a)}$. 
This lower bound from the single-edge Norros-Reittu model turns out to be {\em equal} to the asymptotic size of the giant in the {\em multi-edge} Norros-Reittu model as identified in Theorem \ref{thm:asymp-multi}. 
This shows that the two giants from the multi- and single-edge settings have asymptotically equal sizes in the regime given by \eqref{percn-supercrit-single}.
In Section~\ref{app-barely-supercritical-regime}, we collect all these ingredients to complete the proof of  Theorem \ref{thm:asymp-single}.

\subsection{Size and weight of the giant core}
\label{sec:size-core-giant}
Consider the measure space $\mathcal{S}_a = ((0,a], \sB((0,a]), \Lambda_a)$, where $\sB((0,a])$ denotes the Borel sigma-algebra on $(0,a]$, and $\Lambda_a(\dif x) = \frac{\dif x}{a}$ is the normalized Lebesgue measure on $(0,a]$.
Recall from~\eqref{eq:p-ij-NR-defn} that the probability that there is an edge between $i$ and $j$ after percolation equals 
	$p_{ij} =\percn [1-\e^{-w_{i}w_{j}/\ell_n}].$
For $u,v\in(0,a]$, define the kernel 
	\eqn{
	\label{kappa-Nn-def}
	\kappa_{\sss N_n}^{\sss (a)}(u,v)=N_n(a)  p_{\lceil N_n u \rceil \lceil N_n v \rceil}.
	}
Then putting $u_i^n = i/N_n$, we have that for all $i\in [N_n(a)]$, $p_{ij} =  \kappa_{\sss N_n}^{\sss (a)} (u_i^n,u_j^n)/N_n(a) $. 
Obviously, the empirical measure $\Lambda_{n,a}$ of $(u_i^n)_{i\in [N_n(a)]}$ converges in the weak-topology, with the limiting measure $\Lambda_a$.
This verifies \cite[(2.2)]{BJR07}, and thus  $(\cS_a, ((u_i^n)_{i\in [N_n(a)]})_{n\geq 1})$ is a vertex space according to the definition in \cite[Section 2]{BJR07}. 
Next, we verify that $(\kappa_{\sss N_n}^{\sss (a)})_{n\geq 1}$ is a sequence of graphical kernels on  $\cS_a$
according to \cite[Definition 2.9]{BJR07}. 
By \eqref{eq:asymp-w-N}, 
$w_{\lceil N_n u \rceil}w_{\lceil N_n v \rceil}/\ell_n=o(1).$ 
Using $1-\e^{-x} = x(1+o(1))$ as $x\to 0$, we have that, for any $(u_n)_{n\geq 1}, (v_n)_{n\geq 1} \subset (0,a]$ and $u,v\in (0,a]$ with $u_n \to u$ and $v_n\to v$, 
\eqn{
	\label{kappa-a-def}
	\kappa_{\sss N_n}^{\sss (a)}(u_n,v_n)\rightarrow \kappa^{\sss(a)}(u,v):=  \frac{a \cf^2}{\mu} (uv)^{-\alpha} \quad \text{for all }u,v\in (0,a].
	}
Note that $\kappa^{\sss (a)}$ is bounded a.e., and thus the first two conditions of \cite[Definition~2.7]{BJR07} are satisfied. 
Next, note that for any $a>0$ fixed, 
	\eqan{
	\frac{1}{N_n(a)}\sum_{\substack{i,j\in [N_n(a)]\\ i<j}} \percn [1-\e^{-w_iw_j/\ell_n}]
	&\to \frac{1}{2a}  \int_0^a \int_0^a \frac{\cf^{2}}{\mu}(uv)^{-\alpha}\dif u\dif v\nn\\
	&=\frac{1}{2} \int_0^a \int_0^a  \kappa^{\sss (a)}(u,v) \Lambda_a(\dif u)\Lambda_a(\dif v)<\infty,
	}
which verifies \cite[(2.11)]{BJR07}, and thus all the conditions of \cite[Definition 2.9]{BJR07} have now been verified. 
Finally, $\kappa^{\sss (a)} >0$, so that it is irreducible according to \cite[Definition 2.10]{BJR07}. 
Hence we have verified that $\cG_{\sss N_n(a)}$ is an inhomogeneous random graph with kernels $(\kappa^a_{\sss N_n})_{n\geq 1}$ satisfying all the requisite good properties in \cite{BJR07}. Further, since $\kappa^{\sss(a)}(u,v)$ in \eqref{kappa-a-def} is of product structure, the kernel $\kappa^{\sss(a)}$ is rank-1, as discussed in more detail in \cite[Section 16.4]{BJR07}.

To describe the giant component,  define the integral operator ${\bf T}_{\kappa^{\sss(a)}}\colon L^2(\mathcal{S}_a) \mapsto L^2(\mathcal{S}_a)$ by 
	\eqn{ \label{eq:integral-operator-defn}
	({\bf T}_{\kappa^{\sss(a)}} f)(u) = \int_{0}^a \kappa^{\sss(a)}(u,v)f(v)\Lambda_a(\dif v) = \int_0^a\frac{\cf^{2}}{\mu}(uv)^{-\alpha} f(v) \dif v,
	}
and let $\|{\bf T}_{\kappa^{\sss(a)}}\|$ denote its operator norm. Then the following holds:
\begin{fact}
For all $a>0$, we have $\|{\bf T}_{\kappa^{\sss(a)}}\| = \infty$.
\end{fact}
\begin{proof}
Fix $\varepsilon>0$ and define
\eqn{ \label{eq:integral-operator-defn-ve}
	({\bf T}^\varepsilon_{\kappa^{\sss(a)}} f)(u) = \int_{\varepsilon}^a \frac{\cf^{2}}{\mu}(uv)^{-\alpha} f(v) \dif v.
	}
Then the largest eigenvector of ${\bf T}^{\varepsilon}_{\kappa^{\sss(a)}}$ is proportional to $v^{-\alpha}$ and the largest eigenvalue is proportional to $\int_\varepsilon^a v^{-2\alpha} \dif v $. 
Since $\kappa^{\sss(a)}$ is symmetric, the operator norm is also proportional to $\int_\varepsilon^a v^{-2\alpha} \dif v $. Using the fact that $2\alpha>1$, we get  $\lim_{\varepsilon \to 0}\|{\bf T}^{\varepsilon}_{\kappa^{\sss(a)}}\| = \infty$ for every $a>0$. 
Next, by the Perron-Frobenius theorem, we have  $\|{\bf T}_{\kappa^{\sss(a)}}\| \geq \|{\bf T}^{\varepsilon}_{\kappa^{\sss(a)}}\|$ for all $\varepsilon>0$, and thus the proof follows. 
\end{proof}

Let $\sCa_{\sss (i)}$ denote the size of $i$-th largest component of the graph $\cG_{\sss N_n(a)}$. 
Throughout this section, we suppress $\percn$  in the notation.
To describe the size of the giant component,  let  $\cX_{a}(u)$ be a multi-type branching process with type space $\cS_a$, where we start from one vertex with type $u\in \cS_a$, and a particle of type $v\in \cS_a$ produces progeny in the next generation according to a Poisson process on $\cS_a$ with intensity $ \kappa^{\sss (a)}(v,x) \Lambda_a(\dif x)$.
Let $\rho_a(u)$ be the survival probability of $\cX_{a} (u)$, and $\rho_{a,{\sss \geq k}}$ denote the probability that $\cX_{a} (u)$ has at least $k$ individuals. 
Define 
\begin{eq}\label{defn:survival-prob}
\rho_a = \int_0^a \rho_a (u) \Lambda_a(\dif u) = \frac{1}{a}\int_0^a \rho_a (u) \dif u, \quad  \rho_{a,{\sss \geq k}} = \int_0^a \rho_{a,{\sss \geq k}} (u) \Lambda_a(\dif u) = \frac{1}{a}\int_0^a\rho_{a,{\sss \geq k}}(u) \dif u.  
\end{eq}
The following proposition describes the emergence of the giant component for $\cG_{\sss N_n(a)}$: 
\begin{proposition}[Giant in $\mathcal{G}_{\sss N_n(a)}$]\label{prop:giant-restricted}
Under {\rm Assumption~\ref{assumption-NR}}, for all $a>0$, $|\sC_{\sss (1)}^{a}| = N_n(a)\rho_{a} (1+\oP(1))$, 
and $|\sC_{\sss (2)}^{a}| = \oP(N_n(a))$. 
\end{proposition}
\begin{proof}
The asymptotics of $|\sCa_{\sss (1)}|$  follows  directly by applying
\cite[Theorem 3.1]{BJR07}. 
The asymptotics of $|\sCa_{\sss (2)}|$  follows from \cite[Theorem 3.6]{BJR07}.
\end{proof}

\paragraph{Branching process analysis.}
We next analyse the limiting branching process on $[N_n(a)]$, and its limit as $a\rightarrow \infty$. Recall from \eqref{defn:survival-prob} that $\rho_a (u)$ is the survival probability of an individual of type $u\in[0,a]$. Note that the offspring distribution of a vertex of type $u\in[0,a]$ is Poisson with mean given by the limit of
    \eqn{
    \percn w_{\lceil uN_n\rceil} \sum_{j\in [N_n(a)]} \frac{w_j}{\ell_n}
    \rightarrow \frac{\cf^2}{\mu} u^{-\alpha} \int_0^a v^{-\alpha} \dif v
    =\frac{\cf^2}{\mu(1-\alpha)}u^{-\alpha} a^{1-\alpha}
    \equiv \barcf a^{1-\alpha} u^{-\alpha},
    }
where, by \eqref{mu-comp},
    \eqn{
    \label{bar-cf-def}
    \barcf=\frac{\cf^2}{\mu(1-\alpha)}=\cf.
    }
By the rank-1 structure of the limiting branching process, each of the $\Poi(\barcf a^{1-\alpha} u^{-\alpha})$ children of $u$ has survival probability $\rho_a^\star$ satisfying 
	\eqn{
	\label{surv-prob-comp-app}
	\bar{\rho}_{a}^{\star}=\int_0^a \frac{\cf u^{-\alpha}}{\int_0^a \cf v^{-\alpha}\dif v} [1-\e^{-\barcf a^{1-\alpha} u^{-\alpha}\rho^{\star}_{a}}]\dif u.
	}
We then get that $a^{1-\alpha}\rho_a^{\star}\rightarrow \bar{\rho}^{\star}_{\sss \infty}$, as $a\rightarrow \infty$, where $\bar{\rho}^{\star}_{\sss \infty}$ is the largest solution of
    \eqn{
 	\label{bar-rho-def}
 	\bar{\rho}^{\star}_{\sss \infty}=(1-\alpha)\int_0^{\infty} u^{-\alpha} \big[1-\e^{-\barcf u^{-\alpha}\bar{\rho}^{\star}_{\sss \infty}}\big]\dif u.
 	}
Note that, for $c>0$ and writing $z=c/u^{\alpha}$,
    \eqn{
    \int_0^{\infty} u^{-\alpha} \big[1-\e^{-c u^{-\alpha}}\big]
    \dif u
    =\frac{(\tau-1)c^{\tau-2}}{\tau-2}\Gamma(3-\tau).
    }
Thus, 
\eqref{bar-rho-def} reduces to
    \eqn{
    \bar{\rho}^{\star}_{\sss \infty}= (\barcf \bar{\rho}^{\star}_{\sss \infty})^{\tau-2}\Gamma(3-\tau)
    ,
    }
so that
 	\eqn{
 	\label{rho-star-formula}
 	\bar{\rho}^{\star}_{\sss \infty}
 	=\Gamma(3-\tau)^{1/(3-\tau)} \barcf^{(\tau-2)/(3 -\tau)}.
 	}
Thus, $\rho_a (u)$ satisfies, as $a\rightarrow \infty$,
    \eqn{
    \label{rho-a-u-comp}
    \rho_a (u)=1-\e^{-\barcf a^{1-\alpha} u^{-\alpha}
    \rho_a^{\star}}
    \rightarrow 1-\e^{-\barcf u^{-\alpha}
    \bar{\rho}^{\star}_{\sss \infty}}.
    }
\paragraph{Analysis of total weight inside the giant of the core.}

We conclude this section by providing the asymptotics of the total weight inside  $\sCa_{\sss (1)}$:

\begin{proposition}[Weight of the giant in $\cG_{\sss N_n(a)}$]\label{prop-weight-giant-core}
Under {\rm Assumption~\ref{assumption-NR}}, for any fixed $a>0$ and $\vep>0$, 
    \begin{eq}
    \label{prob-LB-1-NBH}
    \lim_{n\to\infty}\mathbb{P}\Big(\sum_{i\in \sCa_{\sss (1)} } \percn w_i \geq (\zeta_a-\vep)n\percn^{1/(3-\tau)}\Big)=1,
    \end{eq}
where $\zeta_a:=\int_0^a \cf u^{-\alpha} \rho_a(u) \dif u$,
and 
    \eqn{
    \label{zeta-a-limit}
    \lim_{a\rightarrow \infty}
    \zeta_a=\zeta=\mu \kappa^{1/(3-\tau)},
    }
where $\kappa=\cf^{\tau-2} \Gamma(3-\tau)$.
\end{proposition}

\begin{proof}
We apply \cite[Theorem 9.10]{BJR07}. Choose $\delta>0$ sufficiently small.
We drop the contribution due to $i\leq N_n(\delta)$ to obtain
    \eqn{
    \sum_{i\in \sCa_{\sss (1)} } \percn w_i
    \geq \sum_{i\in \sCa_{\sss (1)}\cap [N_n(\delta)]} 
    \percn w_i.
    }
Further, for all $i\in [N_n(a)]\setminus [N_n(\delta)]$, the function $i\mapsto \frac{w_i}{w_{N_n}}$ is bounded. Thus,  \cite[Theorem 9.10]{BJR07} is applicable and we have 
\eqan{
\sum_{i\in \sCa_{\sss (1)}\cap [N_n(\delta)]}\percn w_i &= \big(w_{N_n}\percn N_n(a)\big)
\frac{1}{N_n(a)}\sum_{i\in \sCa_{\sss (1)}\cap [N_n(\delta)]} \frac{w_i}{w_{N_n}}\\
&=(1+\oP(1))\big(w_{N_n}\percn N_n(a)\big)\int_\delta^a u^{-\alpha} \rho_a(u) \Lambda_a(\dif u)\nn\\
&=(1+\oP(1))\big(w_{N_n}\percn N_n\big)\int_\delta^a u^{-\alpha} \rho_a(u)\dif u.\nn
}
Since
    \eqn{
    \int_0^\delta u^{-\alpha} \rho_a(u) \dif u
    \leq \int_0^\delta u^{-\alpha}\dif u\leq \vep/\cf,
    }
by choosing $\delta=\delta(\vep)>0$ sufficiently small, we obtain
    \begin{eq}
    \label{prob-LB-1-NBH-a}
    \lim_{n\to\infty}\mathbb{P}\Big(\sum_{i\in \sCa_{\sss (1)} } \percn w_i \geq (\zeta_a-\vep)w_{N_n}\percn N_n/\cf\Big)=1.
    \end{eq}
The proof of \eqref{prob-LB-1-NBH} follows by noting that \eqref{Nn-supercrit}--\eqref{eq:asymp-w-N} imply that $w_{N_n}\percn N_n=\cf n\percn^{1/(3-\tau)}$. 

We continue with \eqref{zeta-a-limit}, for which we note that
    \eqan{
    \zeta_a=\int_0^a \cf u^{-\alpha} \rho_a(u) \dif u.
 	}
We use \eqref{rho-a-u-comp} to further write this as
    \eqan{
    \zeta_a&=\int_0^a \cf u^{-\alpha}[1-\e^{-\barcf a^{1-\alpha} u^{-\alpha}
    \rho_a^{\star}}]\dif u
    =\int_0^a \cf u^{-\alpha}[1-\e^{-\barcf u^{-\alpha}
    \bar{\rho}^{\star}_{\sss \infty}}]\dif u+o(1)\\
    &\rightarrow \zeta\equiv \int_0^\infty \cf u^{-\alpha}[1-\e^{-\barcf u^{-\alpha}
    \bar{\rho}^{\star}_{\sss \infty}}]\dif u=\frac{\cf \bar{\rho}^{\star}_{\sss \infty}}{1-\alpha}
    =\Gamma(3-\tau)^{1/(3-\tau)} \cf \barcf^{(\tau-2)/(3 -\tau)}\frac{\tau-1}{\tau-2}\nn,
    }
by \eqref{bar-rho-def} and \eqref{rho-star-formula}. By \eqref{mu-comp}, this constant equals 
    \eqn{
    \zeta=\mu \kappa^{1/(3-\tau)},
    }
where $\kappa=\cf^{\tau-2} \Gamma(3-\tau)$ as claimed in \eqref{zeta-a-limit}.
\end{proof}

\subsection{Lower bound using 1-neighborhood}
\label{sec-lower-bound-single-edge}
We next aim to use Proposition \ref{prop-weight-giant-core} to identify the constant in the lower bound on the giant in the barely supercritical regime in the single-edge constrained case as studied in this paper.
The main result is the following proposition, which proves a lower bound on the largest connected component in the percolated $\mNR$:

\begin{proposition}[Size of tiny giant]
\label{prop-LB-size-span-C1}
Under {\rm Assumption~\ref{assumption-NR}}, for any $\vep>0$, as $n\to\infty$, 
    \begin{eq}
    \label{prob-LB-Cmax}
    \mathbb{P}\Big(|\sC_{\sss (1)}^*(\percn)|\geq (\zeta-\vep)n\percn^{1/(3-\tau)}\Big)=1,
    \end{eq}
where $\zeta$ is defined in \eqref{zeta-a-limit}.
\end{proposition}

The key intuition in the proof of Proposition \ref{prop-LB-size-span-C1} is that the size of the giant connected component should be approximately equal to the size of the 1-neighborhood of $\sC_{\sss (1)}^a$, which approximately equals
    \eqn{
    (1+\oP(1))\sum_{j\in \sC_{\sss (1)}^{a}} \percn w_j,
    }
as studied in Proposition \ref{prop-weight-giant-core}. 
Recall that $\beta_n = n \percn ^{1/(3-\tau)}$. 
The following 
lemma shows that the size of the 1-neighborhood ${\mathcal N}_1(\sC_{\sss (1)}^{a})$ of $\sC_{\sss (1)}^{a}$ is closely concentrated around $\sum_{j\in \sC_{\sss (1)}^{a}} \percn w_j$:
\begin{lemma}[Direct neighbors of {$\sC_{\sss (1)}^{a}$}]
\label{prop:2nbd-giant}
For any fixed $a>0$, and $\vep >0$, as $n\to\infty$,
	\begin{eq}
	\label{eq:1-nbd-prob-convergence}
	\PR\bigg(\Big| |{\mathcal N}_1(\sC_{\sss (1)}^{a})| - \sum_{i\in \sC_{\sss (1)}^{a}} \percn w_i \Big| > \vep\beta_n \ \bigg\vert\ \cG_{\sss N_n(a)} \bigg) \pto 0. 
	\end{eq} 
\end{lemma}

\begin{proof}
Let $\PR_1$ and $\E_1$ denote the conditional probability and expectation, respectively, conditionally on $\cG_{\sss N_n(a)}$. 
Let us first show that 
	\begin{eq}\label{eq:expt-weight-main-contribution}
	\E_1\big[\big|{\mathcal N}_1(\sC_{\sss (1)}^{a})\big|\big] =   \sum_{i\in \sC_{\sss (1)}^{a}}  \percn w_i+\oP(\beta_n).
	\end{eq}
Note that 
	\eqn{\label{one-neighbor-sum-indic}
	\big|{\mathcal N}_1(\sC_{\sss (1)}^{a})\big|= \sum_{j\notin [N_n(a)]} \ind{(i,j) \text{ is present for some }i\in \sC_{\sss (1)}^{a}}.
	}
Thus, by a union bound, the expectation in \eqref{eq:expt-weight-main-contribution} is at most 
    \begin{eq}
    \label{eq:sim-split-01}
    \percn \sum_{i\in \sC_{\sss (1)}^{a}}  \sum_{j\notin [N_n(a)]} \big( 1- \e^{-w_iw_j/\ell_n}\big).
    \end{eq}
Moreover, using inclusion-exclusion,  the expectation in \eqref{eq:expt-weight-main-contribution} is at least 
	\begin{eq}\label{eq:sim-split-1}
	&\sum_{i\in \sC_{\sss (1)}^{a}}  \sum_{j\notin [N_n(a)]} \percn \big( 1- \e^{-w_iw_j/\ell_n}\big)   - \percn^2\sum_{i_1,i_2\in \sC_{\sss (1)}^{a}} \sum_{j\notin [N_n(a)]} \big( 1- \e^{- w_{i_1}w_j/\ell_n}\big)\big( 1- \e^{- w_{i_2}w_j/\ell_n}\big).
	\end{eq}
Now, 
    \eqn{
    \sum_{i\in \sC_{\sss (1)}^{a}} w_i \leq \sum_{i\in [N_n(a)]} w_i \leq C a^{1-\alpha} (N_n)^{1-\alpha} n^\alpha.
    }
Using $1-\e^{-x} \leq x$ and \eqref{Nn-supercrit-rew}, the second term in \eqref{eq:sim-split-1} is at most 
    \begin{eq}
    \percn^2 \bigg(\sum_{i\in \sC_{\sss (1)}^{a}} w_i\bigg)^2 \frac{1}{\ell_n^2} \sum_{j\notin [N_n(a)]} w_j^2&\leq Ca^{2-2\alpha}\percn^2 (N_n)^{2-2\alpha} n^{-2+4\alpha} \sum_{j>N_n(a)} j^{-2\alpha} \\
    &\leq C a^{3-4\alpha }\percn^2 n^{-2+4\alpha}(N_n)^{3-4\alpha} = CN_n=o(\beta_n).
    \end{eq}
Therefore, \eqref{eq:sim-split-01} and \eqref{eq:sim-split-1} together imply that 
    \begin{eq}
    \label{eq:one-step-simple}
    \E_1\big[\big|{\mathcal N}_1(\sC_{\sss (1)}^{a})\big|\big] =  \percn \sum_{i\in \sC_{\sss (1)}^{a}} \sum_{j\notin [N_n(a)]}\big( 1- \e^{-w_iw_j/\ell_n}\big) + \oP(\beta_n).
    \end{eq}
It is not hard to see that 
    \begin{eq}
    \label{eq:one-step-simple-sec}
    \sum_{i\in \sC_{\sss (1)}^{a}} \sum_{j\notin [N_n(a)]}\percn \big( 1- \e^{-w_iw_j/\ell_n}\big)=  \sum_{i\in \sC_{\sss (1)}^{a}} \sum_{j\notin [N_n(a)]}\frac{\percn  w_iw_j}{\ell_n} + \oP(\beta_n).
    \end{eq}
Moreover, again by \eqref{Nn-supercrit-rew},
    \eqan{
    \sum_{i\in \sC_{\sss (1)}^{a}} \sum_{j\in [N_n(a)]} \frac{\percn w_iw_j}{\ell_n}
    &\leq \frac{\percn}{\ell_n} \Big(\sum_{i\in [N_n(a)]} w_i\Big)^2\\
    &\leq \percn C a^{2(1-\alpha)} (N_n)^{2(1-\alpha)} n^{2\alpha-1} = O(N_n)= o(\beta_n).\nn
    }
We conclude that
    \begin{eq}
    \label{eq:one-step-simple-third}
    \E_1\big[\big|{\mathcal N}_1(\sC_{\sss (1)}^{a})\big|\big]=  \sum_{i\in \sC_{\sss (1)}^{a}} \sum_{j\in[n]}\percn  \frac{w_iw_j}{\ell_n} + \oP(\beta_n)
    =\sum_{i\in \sC_{\sss (1)}^{a}} \percn  w_i + \oP(\beta_n),
    \end{eq}
and thus \eqref{eq:expt-weight-main-contribution} follows.

To complete the proof of \eqref{eq:1-nbd-prob-convergence}, we apply Chebyshev's inequality for which we need to bound the conditional variance of $|{\mathcal N}_1(\sC_{\sss (1)}^{a})|$. Let $\mathrm{Var}_1$ denote the variance conditionally on  $\cG_{\sss N_n(a)}$.
Note that 
\eqref{one-neighbor-sum-indic} is a sum of conditionally independent indicators, given $\cG_{\sss N_n(a)}$. 
Therefore,
	\begin{eq}
	\mathrm{Var}_1\big(\big|{\mathcal N}_1 (\sC_{\sss (1)}^{a})\big| \big) \leq \E_1\big[\big|{\mathcal N}_1 (\sC_{\sss (1)}^{a})\big|  \big], 
	\end{eq}
and an application of Chebyshev's inequality completes the proof. 
\end{proof}

Now we are ready to complete the proof of Proposition~\ref{prop-LB-size-span-C1}:

\begin{proof}[Proof of Proposition~\ref{prop-LB-size-span-C1}]
Clearly, $|\sC_{\sss (1)}^*(\percn)|\geq |{\mathcal N}_1(\sC_{\sss (1)}^{a})|$. We apply Lemma \ref{prop:2nbd-giant}, and rely on Proposition \ref{prop-weight-giant-core} to estimate $|\cN_1(\sC_{\sss (1)}^{a})|$. Thus, Proposition~\ref{prop-LB-size-span-C1} follows.
\end{proof}

\subsection{Barely supercritical regime with single-edge constraint}
\label{app-barely-supercritical-regime}
In this section, we prove Theorem \ref{thm:asymp-single} by investigating the barely supercritical regime, where the size of the unique largest connected component for $\mNR$ was identified in Theorem~\ref{thm:asymp-multi}. 
We rely on the obvious inequality
    \eqn{
    |\sC_{\sss (1)}(\percn)|\geq |\sC^*_{\sss (1)}(\percn)|.
    } 
By Proposition~\ref{prop-LB-size-span-C1}, with high probability and for every $\vep>0$,
    \eqn{
    |\sC_{\sss (1)}^*(\percn)|\geq (\zeta-\vep)n\percn^{1/(3-\tau)}.
    }
By Theorem~\ref{thm:asymp-multi}, on the other hand, again with high probability and for every $\vep>0$,
    \eqn{
    |\sC_{\sss (1)}(\percn)|\leq (\zeta+\vep)n\percn^{1/(3-\tau)}.
    }
This completes the proof of \eqref{super-critical-equivalent} that shows that for $\percn$, the single-edge constraint is insignificant.

\qed

\paragraph{Acknowledgements.}
SD was partially supported by Vannevar Bush Faculty Fellowship ONR-N00014-20-1-2826.
The work of RvdH is supported in part by the Netherlands Organisation for Scientific Research (NWO) through the Gravitation {\sc NETWORKS} grant no. 024.002.003.

\bibliographystyle{abbrv}
\bibliography{bibliography}

\begin{thebibliography}{10}

\bibitem{AJB00}
R.~Albert, H.~Jeong, and A.-L. Barab{\'{a}}si.
\newblock {Error and attack tolerance of complex networks}.
\newblock {\em Nature}, 406:378, 2000.

\bibitem{Bar16}
A.-L. Barab{\'{a}}si.
\newblock {\em {Network Science}}.
\newblock Cambridge University Press, 1 edition, 2016.

\bibitem{BDH18}
S.~Bhamidi, S.~Dhara, and R.~{\swap{Hofstad}{van der }}.
\newblock {Multiscale genesis of the tiny giant for percolation on scale-free
  random graphs}.
\newblock {\em In Preparation}, 2021+.

\bibitem{BHL12}
S.~Bhamidi, R.~van~der Hofstad, and J.~S.~H. van Leeuwaarden.
\newblock {Novel scaling limits for critical inhomogeneous random graphs}.
\newblock {\em Ann. Probab.}, 40(6):2299--2361, 2012.

\bibitem{BJR07}
B.~Bollob{\'{a}}s, S.~Janson, and O.~Riordan.
\newblock {The phase transition in inhomogeneous random graphs}.
\newblock {\em Random Struct.~Alg.}, 31(1):3--122, 2007.

\bibitem{BR03}
B.~Bollob{\'{a}}s and O.~Riordan.
\newblock {Robustness and vulnerability of scale-free random graphs}.
\newblock {\em Internet Math.}, 1(1):1--35, 2003.

\bibitem{CNSW00}
D.~S. Callaway, M.~E.~J. Newman, S.~H. Strogatz, and D.~J. Watts.
\newblock {Network Robustness and Fragility: Percolation on Random Graphs}.
\newblock {\em Phys. Rev. Lett.}, 85(25):5468--5471, 2000.

\bibitem{CbAH02}
R.~Cohen, D.~Ben-Avraham, and S.~Havlin.
\newblock {Percolation critical exponents in scale-free networks}.
\newblock {\em Phys. Rev. E}, 66(3):36113, 2002.

\bibitem{CEbAH00}
R.~Cohen, K.~Erez, D.~Ben-Avraham, and S.~Havlin.
\newblock {Resilience of the internet to random breakdowns}.
\newblock {\em Phys. Rev. Lett.}, 85(21):4626--4628, nov 2000.

\bibitem{DHL19}
S.~Dhara, R.~van~der Hofstad, and J.~S.~H. van Leeuwaarden.
\newblock {Critical percolation on scale-free random graphs: New universality
  class for the configuration model}.
\newblock {\em Commun. Math. Phys.}, 382:123--171, 2021.

\bibitem{DGM08}
S.~N. Dorogovtsev, A.~V. Goltsev, and J.~F.~F. Mendes.
\newblock {Critical phenomena in complex networks}.
\newblock {\em Rev. Mod. Phys.}, 80(4):1275--1335, 2008.

\bibitem{Hof17}
R.~{\swap{Hofstad}{van der }}.
\newblock {\em {Stochastic Processes on Random Graphs}}.
\newblock Lecture notes for the 47th Summer School in Probability Saint-Flour
  2017, 2017.

\bibitem{HJL16}
R.~{\swap{Hofstad}{van der }}, S.~Janson, and M.~J. Luczak.
\newblock {Component structure of the configuration model: barely supercritical
  case}.
\newblock {\em Random Struct.~Alg.}, 55:3--55, 2019.

\bibitem{Jan94}
S.~Janson.
\newblock {\em {Orthogonal Decompositions and Functional Limit Theorems for
  Random Graph Statistics.}}
\newblock Mem. Amer. Math. Soc., 1994.

\bibitem{JL09}
S.~Janson and M.~J. Luczak.
\newblock {A new approach to the giant component problem}.
\newblock {\em Random Struct.~Alg.}, 34(2):197--216, 2009.

\bibitem{Newman-book}
M.~E.~J. Newman.
\newblock {\em Networks: An introduction}.
\newblock Oxford University Press, Oxford, 2010.

\bibitem{NR06}
I.~Norros and H.~Reittu.
\newblock {On a conditionally Poissonian graph process}.
\newblock {\em Adv. Appl. Probab.}, 38(1):59--75, 2006.

\end{thebibliography}

\end{document}